\documentclass{siamltex}%
\usepackage{amsfonts}
\usepackage{amsmath}
\usepackage{amssymb}
\usepackage{graphicx}
\usepackage{enumerate}
\usepackage{color}
\usepackage{hyperref}
\usepackage[USenglish]{babel}%
\newtheorem{remark}[theorem]{Remark}

\begin{document}

\title{CONTROLLABILITY NEAR A HOMOCLINIC BIFURCATION}
\author{Fritz Colonius\\Institute of Mathematics, University of Augsburg, Augsburg, Germany
\and Amani Hasan\\School of Mathematics, Statistics, and Computer Science, College of Science,
University of Tehran, Iran \& Faculty of Mechanical and Electrical
Engineering, Damascus University, Syria
\and Gholam Reza Rokni Lamouki\\School of Mathematics, Statistics, and Computer Science, College of Science,
University of Tehran, Iran}
\maketitle

\textbf{Abstract:} Controllability properties are studied for control-affine
systems depending on a parameter $\alpha$ and with constrained control values.
The uncontrolled systems in dimension two and three are subject to a
homoclinic bifurcation. This generates two families of control sets depending
on a parameter in the involved vector fields and the size of the control
range. A new parameter $\beta$ given by a split function for the homoclinic
bifurcation determines the behavior of these control sets. It is also shown
that there are parameter regions where the uncontrolled equation has no
periodic orbits, while the controlled systems have periodic solutions
arbitrarily close to the homoclinic orbit.\medskip


\textbf{MSC\ 2020: }93C15, 37G15, 93B05, 34C37

\section{Introduction}

Complete controllability is a rare occurrence for nonlinear systems with
restricted control range. Hence the (maximal) regions in the state space where
complete controllability holds, i.e., control sets, are of interest, cf.
Definition \ref{def:Dset}. A basic reference is Colonius and Kliemann
\cite{ColK00}. The present paper studies control sets for parameter dependent
systems in dimension two and three near a homoclinic bifurcation.

Several results for control sets near local bifurcations are available. For
transcritical and pitchfork bifurcations in the one-dimensional case and for
Hopf bifurcations, cf. \cite[Section 8.3 and Section 9.3]{ColK00}, also for
applications to physically relevant systems and further references. Lamb,
Rasmussen, and Rodrigues \cite{LamRR15} develop a topological bifurcation
theory for minimal invariant sets (which coincide with invariant control sets)
of set-valued dynamical systems. The only contribution for control sets near a
homoclinic bifurcation is due to H\"{a}ckl and Schneider \cite{HS} who study
systems when the uncontrolled two-dimensional system is obtained by the
universal unfolding of a Takens-Bogdanov singularity. The relation of our
results to \cite{HS} is discussed in more detail in Remark \ref{Remark_HS1}
and Remark \ref{Remark_HS2}. Control sets near homoclinic and heteroclinic
orbits are also of relevance in the study of models for ship roll motion, cf.
Gayer \cite{Gayer04, Gayer05} and Colonius, Kreuzer, Marquardt, and Sichermann
\cite{ColKMS2008}. While in the latter references the uncontrolled and
unperturbed system is Hamiltonian, the present paper considers non-Hamiltonian
cases. We use the monograph Kuznetsov \cite{Y1998} as a basic reference for
homoclinic bifurcations, cp. also Guckenheimer and Holmes \cite{GucH83} and
Wiggins \cite{Wiggins1988}.

We consider control-affine systems in $\mathbb{R}^{d}$ of the form%
\begin{equation}
\dot{x}(t)=f_{0}(\alpha,x(t))+\sum_{i=1}^{m}u_{i}(t)f_{i}(\alpha,x(t)),u(t)\in
U\text{ with }0\in U\subset\mathbb{R}^{m}, \label{1.1}%
\end{equation}
with parameter $\alpha\in\mathbb{R}$. The term $u(\cdot)$ can be interpreted
as a control or as a time-dependent deterministic perturbation. The invariant
control sets are also of relevance for the analysis of associated degenerate
Markov diffusion processes, where $u$ is replaced by a random disturbance, cf.
Kliemann \cite{Kli87}. Furthermore, control sets are of interest in connection
with minimal data rates for control systems, since their invariance entropy
can be computed, cf. Kawan and Da Silva \cite{KawaDS16}.

For the uncontrolled system $\dot{x}=f_{0}(\alpha_{0},x)$ and dimension $d=2$,
a classical theorem due to Andronov and Leontovich completely describes the
bifurcation of an orbit homoclinic to an hyperbolic equilibrium $x_{0}$. The
saddle quantity $\sigma_{0}$ (the sum of the eigenvalues of $\frac{\partial
f_{0}}{\partial x}(\alpha_{0},x_{0})$) and the sign of a parameter $\beta$
given by a split function determine the direction of the bifurcation and the
stability properties of the periodic orbits. It turns out that also the
properties of control sets for (\ref{1.1}) are determined by this new
parameter $\beta$ (instead of $\alpha$).\ Furthermore, a main result of this
paper shows that the qualitative behavior of the control system can be
different from the behavior of the uncontrolled system: There are parameter
regions where there is no homoclinic orbit and no limit cycle for the
uncontrolled systems while there exist periodic orbits of the control system
arbitrarily close to the homoclinic orbit. The analysis of homoclinic
bifurcations of systems in $\mathbb{R}^{3}$ goes back to the work by L.P.
Shil'nikov. We will only consider the cases, where unique periodic orbits
bifurcate, the much more complicated case where, in particular, countably many
periodic orbits occur is left for future work.

The contents of this paper are as follows. In Section \ref{Section2}, we
introduce notation used for homoclinic bifurcations and cite relevant results
in dimension two and three. Section \ref{Section3} recalls properties of
control sets and their parameter dependence, when the control range is
perturbed or an external parameter occurs in the vector fields. Section
\ref{Section4} starts with a discussion of the control sets near an orbit
homoclinic to a hyperbolic equilibrium in dimension $d=2$. In dimension $d=3$,
we analyze the cases where the equilibrium is a saddle, and a saddle-focus
with saddle quantity $\sigma_{0}<0$. Section \ref{Section5} presents an
example including numerical results which are based on H\"{a}ckl's algorithm
(H\"{a}ckl \cite{G1992}). We remark that an alternative for computing control
sets are set oriented methods, cf. Szolnoki \cite{Szol03}. Finally, Section
\ref{Section6} draws some conclusions.

\textbf{Notation.} The Hausdorff distance between two compact subsets
$A,B\subset\mathbb{R}^{d}$ is $d_{H}(A,B)=\max(\max_{a\in A}\min\{\left\Vert
a-b\right\Vert \left\vert b\in B\right.  \},\max_{b\in B}\min\{\left\Vert
a-b\right\Vert \left\vert a\in A\right.  \})$. The ball of radius $\delta>0$
around $x\in\mathbb{R}^{d}$ is $\mathbf{B}(x,\delta)=\{y\in\mathbb{R}%
^{d}\left\vert \left\Vert x-y\right\Vert <\delta\right.  \}$. It is convenient
to write (as Kuznetsov \cite{Y1998}) $\Gamma_{0}\cup x_{0}$ for the union of
$\{x_{0}\}$ with an orbit $\Gamma_{0}$ homoclinic to $x_{0}$.

\section{Bifurcation of orbits homoclinic to hyperbolic
equilibria\label{Section2}}

This section introduces some notation and cites results on the bifurcation of
orbits which are homoclinic to hyperbolic equilibria. This is done for planar
systems in the first subsection and for three-dimensional systems in the
second subsection \ We rely on the presentation in Kuznetsov \cite[Chapter
6]{Y1998}.

Consider a parameter dependent family of ordinary differential equations in
$\mathbb{R}^{d}$ of the form
\begin{equation}
\dot{x}(t)=f(\alpha,x(t)), \label{2.1}%
\end{equation}
where $f:\mathbb{R}\times\mathbb{R}^{d}\rightarrow\mathbb{R}^{d}$ is a smooth
($C^{\infty}$-)function. We assume that for every $\alpha\in\mathbb{R}$ and
every initial value $x\in\mathbb{R}^{d}$ there exists a unique solution
$\psi^{\alpha}(t,x),t\in\mathbb{R}$, and that all maps $\psi^{\alpha}%
(t,\cdot):\mathbb{R}^{d}\rightarrow\mathbb{R}^{d},t\in\mathbb{R}$, are
continuous. An orbit $\Gamma_{x_{0}}^{\alpha}:=\{\psi^{\alpha}(t,x)\left\vert
t\in\mathbb{R}\right.  \}$ is called homoclinic to an equilibrium point
$x_{0}$ (i.e., $f(\alpha,x_{0})=0$) if $\psi^{\alpha}(t,x)\rightarrow x_{0}$
as $t\rightarrow\pm\infty$. Let
\begin{align*}
W^{\alpha,s}(x_{0})  &  =\{y\in\mathbb{R}^{d}\left\vert \psi^{\alpha
}(t,y)\rightarrow x_{0}\text{ for }t\rightarrow\infty\right.  \},\\
W^{\alpha,u}(x_{0})  &  =\{y\in\mathbb{R}^{d}\left\vert \psi^{\alpha
}(t,y)\rightarrow x_{0}\text{ for }t\rightarrow-\infty\right.  \},
\end{align*}
denote the stable and the unstable manifold, resp., of $x_{0}$.

\subsection{The planar case}

In this subsection we cite a classical theorem due to Andronov and Leontovich
on the bifurcation in the plane of orbits which are homoclinic to hyperbolic equilibria.

Suppose that system (\ref{2.1}) is planar ($d=2$) having for $\alpha_{0}=0$ a
saddle equilibrium $x_{0}=0$, i.e., $f_{x}(0,0)=\frac{\partial}{\partial
x}f(0,0)$ has a positive and a negative eigenvalue, $\lambda_{1}%
(0)<0<\lambda_{2}(0)$. Assume that $\Gamma_{0}$ is an orbit which is
homoclinic to $x_{0}$. For $\alpha$ sufficiently close to $\alpha_{0}=0$, the
implicit function theorem implies that there are saddle equilibria $x_{\alpha
}$ with eigenvalues $\lambda_{1}(x_{\alpha})<0<\lambda_{2}(x_{\alpha})$
depending continuously on $\alpha$.

Let $\Sigma$ be a (one-dimensional) local cross-section to the stable manifold
near the saddle. Select a coordinate $\xi\in\mathbb{R}$ along $\Sigma$ such
that the point of its intersection with the stable manifold $W^{\alpha_{0}%
,s}(x_{0})$ corresponds to $\xi=0$. This coincides with the point of
intersection with the unstable manifold $W^{\alpha_{0},u}(x_{0})=W^{\alpha
_{0},s}(x_{0})=\Gamma_{0}$. For all $\alpha$ sufficiently close to $\alpha
_{0}=0$, $\Sigma$ is also a local transversal section to the unstable
manifolds $W^{\alpha,u}(x_{\alpha})$. Denote by $\xi^{u}(\alpha)$ the $\xi
$-value of the intersection of $W^{\alpha,u}(x_{\alpha})$ with $\Sigma$. The
scalar function $\alpha\mapsto\beta(\alpha):=\xi^{u}(\alpha)$ which is defined
on a neighborhood of $\alpha_{0}=0$ is called a \emph{split function}. The
function $\beta(\cdot)$ is smooth and it is injective if $\beta^{\prime
}(0)\not =0$.

In the planar case considered here, the homoclinic bifurcation is
characterized by the following theorem due to Andronov and Leontovich, cf.
Kuznetsov \cite[Theorem 6.1]{Y1998}. Other references include Guckenheimer and
Holmes \cite[Theorem 6.1.1]{GucH83}, Wiggins \cite[Theorem 3.2.11]%
{Wiggins1988}.

\begin{theorem}
\label{th:Andr}Consider a parameter dependent two-dimensional system of the
form (\ref{2.1}) having at $\alpha_{0}=0$ an orbit $\Gamma_{0}$ which is
homoclinic to a saddle $x_{0}=0$ with eigenvalues $\lambda_{1}(0)<0<\lambda
_{2}(0)$. Assume that the following conditions hold:

(H1)\thinspace\ $\sigma_{0}=\lambda_{1}(0)+\lambda_{2}(0)\neq0$;

(H2)\thinspace\ $\beta^{^{\prime}}(0)\neq0$, where $\beta(\alpha)$ is a split function.

Then, there exist $\overline{\alpha}>0$ and a neighborhood $U_{0}$ of
$\Gamma_{0}\cup x_{0}$ in which for all $\left\vert \alpha\right\vert
<\overline{\alpha}$ a unique limit cycle $L_{\beta(\alpha)}$ bifurcates from
$\Gamma_{0}$. The limit cycle exists and is asymptotically stable for
$\beta>0$ if $\sigma_{0}<0$, and exists and is unstable for $\beta<0$ if
$\sigma_{0}>0$.
\end{theorem}

(H1) is a nondegeneracy condition. Due to (H2) the split function
$\alpha\mapsto\beta(\alpha)$ is injective for $\left\vert \alpha\right\vert $
small enough, hence the inverse $\alpha(\beta)$ exists for $\beta$ in a
neighborhood of $0$ and $\beta$ can be considered as a new parameter. Thus the
unique limit cycle $L_{\beta}$ exists for sufficiently small $\left\vert
\beta\right\vert $. The homoclinic orbit is called \textquotedblleft splitting
down\textquotedblright\ if $\beta<0$ and \textquotedblleft splitting
up\textquotedblright\ if $\beta>0$. We remark (cf. Kuznetsov \cite[formula
(6.25) on p. 232]{Y1998}) that $\beta^{^{\prime}}(0)\neq0$ is equivalent to
the Melnikov condition%
\begin{equation}
M_{\alpha_{0}}(0)=\int_{-\infty}^{+\infty}\ \mathrm{exp}\,\left[  -\int
_{0}^{t}\left(  \frac{\partial{f_{1}}}{\partial{x_{1}}}+\frac{\partial{f_{2}}%
}{\partial{x_{2}}}\right)  d\tau\right]  \,\left(  f_{1}\frac{\partial f_{2}%
}{\partial\alpha}-f_{2}\frac{\partial f_{1}}{\partial\alpha}\right)
\,dt\neq0, \label{M6.25}%
\end{equation}
where all expressions involving $f(0,x_{1},x_{2})=(f_{1}(0,x_{1},x_{2}%
),f_{2}(0,x_{1},x_{2}))^{\top}$ are evaluated along the homoclinic solution of
(\ref{2.1}) at $\alpha_{0}=0$. Hence (H2) is a transversality condition for
the intersection of the stable and unstable manifolds.

\subsection{The three-dimensional case}

As exposed in Kuznetsov \cite[Section 6.3]{Y1998} a three-dimensional state
space gives rise to a wider variety of homoclinic bifurcations. We will
discuss results for hyperbolic equilibria which are saddles and saddle-foci.
Taking into account also the sign of $\sigma_{0}$ there are four main cases,
cf. \cite[p. 214]{Y1998}. We will only treat the three simpler cases.

Consider an equation in $\mathbb{R}^{3}$ of the form (\ref{2.1}) having at
$\alpha_{0}=0$ an orbit $\Gamma_{0}$ homoclinic to a hyperbolic equilibrium
point $x_{0}=0$. It is also possible to define a split function in this case,
cf. \cite[p. 199]{Y1998}. Suppose that the unstable manifold $W^{u}$ of
$x_{0}$ is one-dimensional, introduce a two-dimensional cross-section $\Sigma$
and let the point $\xi^{u}$ correspond to the intersection of $W^{u}$ with
$\Sigma$. Then a split function $\beta=\xi^{u}$ can be defined as before in
the planar case. Its zero $\beta=0$ gives a condition for a homoclinic
bifurcation in $\mathbb{R}^{3}$.

The case of a saddle is described in \cite[Theorem 6.3 and Theorem 6.5]{Y1998}
as follows.

\begin{theorem}
\label{Theorem2.2}Consider system (\ref{2.1}) in $\mathbb{R}^{3}$ having at
$\alpha_{0}=0$ an orbit $\Gamma_{0}$ homoclinic to a saddle $x_{0}=0$ with
real eigenvalues $\lambda_{1}(0)>0>\lambda_{2}(0)>\lambda_{3}(0)$. Assume that
the following conditions hold:

(H1) $\Gamma_{0}$ returns to $x_{0}$ along the eigenspace for $\lambda_{2}(0)$;

(H2) $\beta^{\prime}(0)\not =0$, where $\beta(\alpha)$ is a split function.

(i) Suppose that $\sigma_{0}=\lambda_{1}(0)+\lambda_{2}(0)<0$. Then, there
exists a neighborhood $U_{0}$ of $\Gamma_{0}\cup x_{0}$ in which the system
has a unique and asymptotically stable limit cycle $L_{\beta}$ for all
sufficiently small $\beta>0$. There are no periodic orbits if $\beta\leq0$.

(ii) Suppose that $\sigma_{0}=\lambda_{1}(0)+\lambda_{2}(0)>0$ and,
additionally, $\Gamma_{0}$ is simple or twisted. Then there exists a
neighborhood $U_{0}$ of $\Gamma_{0}\cup x_{0}$ in which for all sufficiently
small $\left\vert \beta\right\vert $ a unique saddle limit cycle $L_{\beta}$
bifurcates from $\Gamma_{0}$. The cycle exists for $\beta<0$ if $\Gamma_{0}$
is simple, and for $\beta>0$ if $\Gamma_{0}$ is twisted. In the first case
there are no periodic orbits if $\beta\geq0$ and in the second case there are
no periodic orbits if $\beta\leq0$.
\end{theorem}

The assumption in (ii) needs some explanation. Here we suppose that the
two-dimen\-sional stable manifold $W^{\alpha_{0},s}(x_{0})$ intersects itself
near the saddle along the two exceptional orbits on $W^{\alpha_{0},s}(x_{0})$
that approach the saddle along the eigenspace for $\lambda_{3}(0)$ (this is
called the \emph{strong inclination property}). This yields a two-dimensional
nonsmooth submanifold which is topologically equivalent to either a simple
band or a twisted band called a M\"{o}bius band (cf. also Wiggins
\cite[Section 4.8A]{S2003}). In the first case, $\Gamma_{0}$ is called simple,
in the second case twisted.

The case of a saddle-focus with $\sigma_{0}<0$ is described in \cite[Theorem
6.4]{Y1998} as follows.

\begin{theorem}
\label{Theorem2.3}Consider system (\ref{2.1}) in $\mathbb{R}^{3}$ having at
$\alpha_{0}=0$ an orbit $\Gamma_{0}$ homoclinic to a saddle-focus $x_{0}=0$
with eigenvalues satisfying $\lambda_{1}(0)>0>\operatorname{Re}\lambda
_{2,3}(0)$ and $\lambda_{2}(0)\not =\lambda_{3}(0)$. Assume that the following
conditions hold:

(H1) $\beta^{\prime}(0)\not =0$, where $\beta(\alpha)$ is a split function;

(H2) $\sigma_{0}=\lambda_{1}(0)+\operatorname{Re}\lambda_{2,3}(0)<0$.

Then exists a neighborhood $U_{0}$ of $\Gamma_{0}\cup x_{0}$ in which the
system has a unique and asymptotically stable limit cycle $L_{\beta}$ for all
sufficiently small $\beta>0$. There are no periodic orbits if $\beta\leq0$.
\end{theorem}

The remaining case of a saddle-focus with $\sigma_{0}>0$ is much more
complicated and leads, among others, to infinitely many saddle limit cycles,
cf. \cite[Theorem 6.6]{Y1998}.

\section{Control sets and their parameter dependence\label{Section3}}

We consider control-affine systems in $\mathbb{R}^{d}$ of the form%
\begin{align}
\dot{x}(t)  &  =f_{0}(x(t))+\sum_{i=1}^{m}u_{i}(t)f_{i}(x(t)),\label{3.1}\\
u  &  \in\mathcal{U}:=\left\{  u\in L^{\infty}(\mathbb{R},\mathbb{R}%
^{m})\left\vert u(t)\in U\,\text{for\ almost all}\,t\in\mathbb{R}\right.
\right\}  ,\nonumber
\end{align}
where $f_{0},f_{1},\ldots,f_{m}$ are smooth vector fields on $\mathbb{R}^{d}$
and the control range $U\subset\mathbb{R}^{m}$ is compact and convex with
$0\in\mathrm{int}U$. We assume that for every initial state $x\in
\mathbb{R}^{d}$ and every control function $u\in\mathcal{U}$ there exists a
unique solution $\varphi(t,x,u),t\in\mathbb{R}$, with $\varphi(0,x,u)=x$ of
(\ref{3.1}) depending continuously on $x$. The system with $u\equiv0$ given by%
\begin{equation}
\dot{x}(t)=f_{0}(x(t)) \label{uncontrolled}%
\end{equation}
is called the uncontrolled system. It generates a continuous flow
$\psi(t,\cdot),t\in\mathbb{R}$, on $\mathbb{R}^{d}$.

The set of points reachable from $x\in\mathbb{R}^{d}$ and controllable to
$x\in\mathbb{R}^{d}$ up to time $T>0$ are defined by
\begin{align*}
{\mathcal{O}}_{\leq T}^{+}(x)  &  :=\{y\in\mathbb{R}^{d}\left\vert {}\right.
\;\text{there are}\;0\leq t\leq T\;\text{and}\;u\in\mathcal{U}\;\text{with}%
\;y=\varphi(t,x,u)\},\\
{\mathcal{O}}_{\leq T}^{-}(x)  &  :=\{y\in\mathbb{R}^{d}\left\vert {}\right.
\;\text{there are}\;0\leq t\leq T\;\text{and}\;u\in\mathcal{U}\;\text{with}%
\;x=\varphi(t,y,u)\},
\end{align*}
resp. Furthermore, the reachable set (or positive orbit) from $x$ and the set
controllable to $x$ (or negative orbit of $x$) are%
\[
\mathcal{O}^{+}(x)=\bigcup\nolimits_{T>0}O_{\leq T}^{+}(x),\quad
\mathcal{O}^{-}(x)=\bigcup\nolimits_{T>0}O_{\leq T}^{-}(x),
\]
resp. The system is called locally accessible in $x$, if $\mathcal{O}_{\leq
T}^{+}(x)$ and $\mathcal{O}_{\leq T}^{-}(x)$ have nonvoid interior for all
$T>0$. This is guaranteed by the accessibility rank condition%
\begin{equation}
\dim\mathcal{LA}\left\{  f_{0},f_{1},\ldots,f_{m}\right\}  (x)=d\text{ for all
}x\in\mathbb{R}^{d}, \label{ARC}%
\end{equation}
where the left hand side denotes the dimension of the subspace of
$\mathbb{R}^{d}$ corresponding to the vector fields evaluated in $x$ in the
Lie algebra $\mathcal{LA}\left\{  f_{0},f_{1},\ldots,f_{m}\right\}  $
generated by the vector fields $f_{0},f_{1},\ldots,f_{m}$ (cf. Sontag
\cite[Theorem 9, p. 156]{Son98}).

The following definition introduces subsets of complete approximate
controllability which are of primary interest in the present paper.

\begin{definition}
\label{def:Dset}A set $D\subset\mathbb{R}^{d}$ is called a control set of
system (\ref{3.1}) if it has the following properties: (i) for all $x\in D$
there is a control function $u\in\mathcal{U}$ such that $\varphi(t,x,u)\in D$
for all $t\geq0$, (ii) for all $x\in D$ one has $D\subset\mathrm{cl}%
\mathcal{O}^{+}(x)$, and (iii) $D$ is maximal with these properties, that is,
if $D^{\prime}\supset D$ satisfies conditions (i) and (ii), then $D^{^{\prime
}}=D$.

A control set $D\subset\mathbb{R}^{d}$ is called an invariant control set if
$\mathrm{cl}D=\mathrm{cl}\mathcal{O}^{+}(x)$ for all $x\in D$. All other
control sets are called variant.
\end{definition}

If the intersection of two control sets is nonvoid, the maximality property
(iii) implies that they coincide. If the system is locally accessible in all
$y\in\mathrm{int}D$, then $\mathrm{int}D\subset\mathcal{O}^{+}(x)$ for all
$x\in D$ and $D=\mathrm{cl}\mathcal{O}^{+}(x)\cap\mathcal{O}^{-}(y)$ for all
$x,y\in\mathrm{int}D$. For these properties and further discussion of control
sets, we refer to Colonius and Kliemann \cite[Chapters 3 and 4]{ColK00}.

Next we will discuss the dependence of control sets on parameters. The
parameters change the size of the control range or the involved vector fields.
First we analyze families of control systems of the form%
\begin{equation}
\dot{x}(t)=f_{0}(x(t))+\sum_{i=1}^{m}u_{i}(t)f_{i}(x(t)),\quad u(t)\in
U^{\rho}:=\rho U, \label{control_aff}%
\end{equation}
where $\rho>0$ and $u\in\mathcal{U}^{\rho}:=\{u\in L^{\infty}(\mathbb{R}%
,\mathbb{R}^{m})\left\vert u(t)\in U^{\rho}\,\text{for\ almost all}%
\,t\in\mathbb{R}\right.  \}$. We suppose that the assumptions on (\ref{3.1})
are satisfied. Obviously, the accessibility rank condition (\ref{ARC}) is
independent of $\rho>0$.

A subset $K\subset\mathbb{R}^{d}$ is called invariant for the uncontrolled
system (\ref{uncontrolled}) if $\psi(t,x)\in K$ for all $x\in K$ and
$t\in\mathbb{R}$. An invariant set $K\subset\mathbb{R}^{d}$ is called chain
transitive if for all $x,y\in K$ and every $\varepsilon>0$ and $T>0$, there
exist $n\in\mathbb{N}$, points $x=x_{0},x_{1},\dots,x_{n}=y\in K$ and times
$t_{0},\dots,t_{n-1}>T$ such that $d(\psi(t_{i},x_{i}),x_{i+1})<\varepsilon$
for $i=0,\dots,n-1$. It is easy to show that an equilibrium, a limit cycle,
and an orbit homoclinic to an equilibrium $x_{0}$ together with $x_{0}$ are
compact chain transitive sets, but they need not be maximal (with respect to inclusion).

The following result describes the behavior of control sets for small control ranges.

\begin{theorem}
\label{Theorem_inner}Consider a family of control-affine systems of the form
(\ref{control_aff}). Let $K\subset\mathbb{R}^{d}$ be a compact maximal chain
transitive set for the flow of the uncontrolled system (\ref{uncontrolled}),
assume that the accessibility rank condition (\ref{ARC}) holds, and the
following inner pair condition holds for all $(x,0)\in K\times\mathcal{U}$:
there is $T>0$ with $\psi(T,x)=\varphi(T,x,0)\in\mathrm{int}\mathcal{O}%
^{+}(x)$. Then there is an increasing family of control sets $D^{\rho}$ of
(\ref{control_aff}) with parameter $\rho>0$ such that
\[
K\subset\mathrm{int}D^{\rho}\text{ and }K=\bigcap\nolimits_{\rho>0}D^{\rho}.
\]
If $K$ is an asymptotically stable equilibrium or periodic orbit, then the
control sets are invariant for $\rho>0$, small enough.
\end{theorem}

\begin{proof}
The first assertion is proved in Colonius and Kliemann \cite[Corollary
4.7.2]{ColK00}. The invariance of the control sets follows from
\cite[Corollary 4.1.13]{ColK00}.
\end{proof}

By \cite[Proposition 4.5.19]{ColK00}, the inner pair condition in $(x,0)$ is
satisfied, if for some $T>0$ the following condition holds in $y=\varphi
(T,x,0)$:%
\begin{equation}
\mathrm{span}\{f_{0}(y),\mathrm{ad}_{f_{0}}^{k}f_{i}(y)\left\vert
i=1,\ldots,m,~k=0,1,\ldots\right.  \}=\mathbb{R}^{d}. \label{NamA}%
\end{equation}
Here the $\mathrm{ad}$-operator is given by iterated Lie brackets,
$\mathrm{ad}_{f_{0}}^{0}f_{i}=f_{i}$ and $\mathrm{ad}_{f_{0}}^{k+1}%
f_{i}=[f_{0},\mathrm{ad}_{f_{0}}^{k}f_{i}]$ for $k\geq0$.

Further results on the dependence of control sets, in particular, their
boundaries, on the parameter $\rho$ are given in Gayer \cite{Gayer04}.

Next we analyze the behavior of control sets under changes of an external
parameter $\alpha$. Consider the following family of control systems on
$\mathbb{R}^{d}$ with $\alpha\in A\subset\mathbb{R}^{k}$,
\begin{equation}
\dot{x}(t)=f_{0}(\alpha,x(t))+\sum_{i=1}^{m}u_{i}(t)f_{i}(\alpha,x(t)),\quad
u\in\mathcal{U}, \label{parameter}%
\end{equation}
with smooth maps $f_{i}:\mathbb{R}^{k}\times\mathbb{R}^{d}\rightarrow
\mathbb{R}^{d},\,i\in\{0,1,\dots,m\}$. We assume that for every $\alpha\in A$
the system satisfies the assumptions on (\ref{3.1}). For $x\in\mathbb{R}^{d}$
and $u\in\mathcal{U}$ the solutions are denoted by $\varphi^{\alpha
}(t,x,u),t\in\mathbb{R}$.

The following theorem describes how the control sets change under parameter
variation. Recall that a set-valued map $x\mapsto F(x)$ between metric spaces
is lower semicontinuous at a point $x_{0}$ if for every open set $O$ with
$F(x_{0})\cap O\not =\varnothing$ it follows that $F(x)\cap O\not =%
\varnothing$ for all $x$ in a neighborhood of $x_{0}$; cf. Aubin and
Frankowska \cite[Definition 1.4.2]{AubF90}.

\begin{theorem}
\label{Theorem_parameter1}For the family of systems (\ref{parameter}) fix a
parameter value $\alpha_{0}\in\mathrm{int}A$. Assume that with $\alpha_{0}$
the accessibility rank condition (\ref{ARC}) is fulfilled and consider a
control set $D^{\alpha_{0}}$ .

(i) Let $K\subset\mathrm{int}D^{\alpha_{0}}$ be a compact set. Then there is
$\delta_{K}>0$ such that for all $\alpha$ with $\left\Vert \alpha-\alpha
_{0}\right\Vert <\delta_{K}$ there is a unique control set $D_{K}^{\alpha}$
with $K\subset\mathrm{int}D_{K}^{\alpha}$ for system (\ref{parameter}) with
parameter value $\alpha$.

(ii) There are $\delta_{0}>0$ and a unique family of control sets $D^{\alpha}$
for all $\alpha$ with $\left\Vert \alpha-\alpha_{0}\right\Vert <\delta_{0}$
with the following property: For every compact set $K\subset\mathrm{int}%
D^{\alpha_{0}}$ there is a $\delta_{K}\in(0,\delta_{0})$ so that
$K\subset\mathrm{int}D^{\alpha}$ for every $\alpha$ with $\left\Vert
\alpha-\alpha_{0}\right\Vert <\delta_{K}$. The set-valued maps $\alpha\mapsto
D^{\alpha}$ and $\alpha\mapsto\mathrm{cl}D^{\alpha}$ are lower semicontinuous
at $\alpha=\alpha_{0}$.
\end{theorem}

This is a special case of Colonius and Lettau \cite[Theorem 3.6]{ColL16} (in
our case, the \textquotedblleft worlds\textquotedblright\ $W^{\alpha
}=\mathbb{R}^{d}$).

\begin{remark}
\label{Remark_near}The proof of Theorem \ref{Theorem_parameter1}(i) provides
the following more precise information. Let $\varphi^{\alpha_{0}}(T,x,u)=y$
for $x,y\in K$. Then for every $\varepsilon>0$ the trajectories of the system
with parameter $\alpha$ satisfying $\varphi^{\alpha}(T^{\alpha},x,u^{\alpha
})=y$ may be chosen with Hausdorff distance $d_{H}(\{\varphi^{\alpha
}(t,x,u^{\alpha})\left\vert t\in\lbrack0,T^{\alpha}]\right.  \},\{\varphi
^{\alpha_{0}}(t,x,u)\left\vert t\in\lbrack0,T]\right.  \})<\varepsilon$ for
$\left\Vert \alpha-\alpha_{0}\right\Vert <\delta$. Here $\delta$ may be chosen
independently of $x,y\in K$.
\end{remark}

The following definition of local control sets replaces the global maximality
property of control sets by a local property, cf. Colonius and Spadini
\cite[Definition 2.2]{ColS03} slightly generalized here.

\begin{definition}
\label{Definition_local}A bounded set $D_{loc}\subset\mathbb{R}^{d}$ is called
a local control set of system (\ref{3.1}) if there exists a neighborhood $V$
of $\mathrm{cl}D_{loc}$ with the following properties: (i) for all $x\in
D_{loc}$ there is a control $u\in\mathcal{U}$ such that $\varphi(t,x,u)\in
D_{loc}$ for all $t\geq0$, (ii) for all $x,y\in D_{loc}$ there exist $T>0$ and
a control $u$ such that $\varphi(t,x,u)\in V$ for all $t\in\lbrack0,T]$ and
$d(\varphi(T,x,u),y)<\varepsilon$, and (iii) $D_{loc}$ is maximal with these properties.
\end{definition}

The results above for control sets remain valid for local control sets. In the
proofs, one simply has to restrict the attention to the isolating neighborhood
$V$ of $\mathrm{cl}D_{loc}$.

The linearization of (\ref{control_aff}) in an equilibrium $(x_{0}%
,0)\in\mathbb{R}^{d}\times\mathbb{R}^{m}$ with $0=f_{0}(x_{0})$ is the control
system%
\begin{equation}
\dot{y}(t)=Ay(t)+Bv(t)\text{ with }A:=\frac{df_{0}(x_{0})}{dx},\quad
B:=\left[  f_{1}(x_{0}),\ldots,f_{m}(x_{0})\right]  . \label{lin}%
\end{equation}
This system is controllable if and only if $\mathrm{rank}[B~AB~\ldots
~A^{d-1}B]=d$.

Local control sets with small control ranges satisfy the following uniqueness
property, cf. \cite[Theorem 5.1]{ColS03}.

\begin{theorem}
\label{Theorem3.6}Consider\ for a family of control-affine systems of the form
(\ref{control_aff}) a hyperbolic equilibrium $x_{0}$ of the uncontrolled
system (\ref{uncontrolled}) and assume that the system linearized in
$(x_{0},0)\in\mathbb{R}^{d}\times\mathbb{R}^{m}$ is controllable. Then there
exist $\rho_{0}>0$ and $\delta_{0}>0$ such that for all $\rho\in(0,\rho_{0})$
the ball $\mathbf{B}(x_{0},\delta_{0})$ contains exactly one local control set
$D_{loc}^{\rho}$ with nonvoid interior.
\end{theorem}

\section{Controllability near homoclinic bifurcations\label{Section4}}

In this section we analyze the control sets that occur near a homoclinic
bifurcation of the uncontrolled system. We consider control-affine systems in
$\mathbb{R}^{2}$ and $\mathbb{R}^{3}$ of the form%
\begin{equation}
\dot{x}(t)=f_{0}(\alpha,x(t))+\sum_{i=1}^{m}u_{i}(t)f_{i}(\alpha,x(t)),\quad
u(t)\in U^{\rho}:=\rho U, \label{4.1}%
\end{equation}
where $\alpha\in A\subset\mathbb{R}$, $0\in\mathrm{int}U$ with $U\subset
\mathbb{R}^{m}$ compact and convex, and $\rho>0$. We assume that for every
$\alpha\in A$ the system satisfies the assumptions on (\ref{3.1}).

These systems depend on the two parameters $(\alpha,\rho)\in A\times
(0,\infty)$. The corresponding control sets will be denoted by $D^{\alpha
,\rho}$ and an analogous notion is used for all other objects. The dependence
of control sets on the parameter $\alpha$ is described in Theorem
\ref{Theorem_parameter1} and the dependence on $\rho$ is described in Theorem
\ref{Theorem_inner}. Throughout this section, the nominal parameter value
$\alpha_{0}$ will be taken as $\alpha_{0}=0$.

\subsection{The planar case}

The following theorem analyzes the control sets when the uncontrolled planar
system undergoes a homoclinic bifurcation in the sense of Theorem
\ref{th:Andr}, and hence for $\alpha_{0}=0$ it has an orbit $\Gamma_{0}$
homoclinic to a saddle equilibrium point $x_{0}=0$ with saddle quantity
$\sigma_{0}=\lambda_{1}(0)+\lambda_{2}(0)\not =0$, a split function
$\beta(\alpha)$ with $\beta^{\prime}(0)\not =0$, and bifurcating limit cycles
$L_{\beta}$. Recall that we may write $\alpha=\alpha(\beta)$ for $\left\vert
\beta\right\vert $ small enough and $\alpha(0)=0$. We use the notation from
Theorem \ref{th:Andr} and, more explicitly, we assume that the limit cycle
$L_{\beta}$ exists for $0<\left\vert \beta\right\vert <\overline{\beta}$ or,
equivalently, for $0<\left\vert \alpha\right\vert <\overline{\alpha}%
:=\alpha(\overline{\beta})$. The following theorem shows that here two
families of control sets are generated depending on the two parameters $\rho$
and $\beta$.

\begin{theorem}
\label{Theorem4.1}Consider a two-parameter family of control-affine systems in
$\mathbb{R}^{2}$ of the form (\ref{4.1}). Suppose that the uncontrolled and
unperturbed system $\dot{x}=f_{0}(0,x)$ satisfies the assumptions of Theorem
\ref{th:Andr} and $\Gamma_{0}\cup x_{0}$ is a maximal chain transitive set.
Furthermore, assume that the accessibility rank condition (\ref{ARC}) holds
for $\alpha_{0}=0$ and the following inner pair condition holds for all
$\beta$ with $0<\left\vert \beta\right\vert <\overline{\beta}$ and all
$x\in\mathbb{R}^{2}$:%
\begin{equation}
\text{For all }\rho>0\text{ there is }T>0\text{ such that }\varphi
^{\alpha(\beta)}(T,x,0)\in\mathrm{int}\mathcal{O}^{\alpha(\beta),\rho,+}(x).
\label{pair}%
\end{equation}

(i) Then there is a family of control sets $D_{0}^{\alpha(\beta),\rho}$,
defined for $\rho>0$ and $\beta\in\left(  -\beta_{0}(\rho),\beta_{0}%
(\rho)\right)  $ with $\beta_{0}(\rho)\in(0,\overline{\beta})$, satisfying for
all $\rho$ and $\beta$%
\begin{equation}
\Gamma_{0}\cup x_{0}\subset\mathrm{int}D_{0}^{\alpha(\beta),\rho}\text{ and
}\Gamma_{0}\cup x_{0}=\bigcap\nolimits_{\rho>0}D_{0}^{0,\rho}. \label{4.1b}%
\end{equation}

(ii) If $\sigma_{0}<0$ there is a family of control sets $D_{1}^{\alpha
(\beta),\rho}$, defined for $\rho>0$ and $\beta\in(0,\overline{\beta})$,
satisfying for all $\rho$ and $\beta$%
\begin{equation}
L_{\beta}\subset\mathrm{int}D_{1}^{\alpha(\beta),\rho}\text{ and }L_{\beta
}=\bigcap\nolimits_{\rho>0}D_{1}^{\alpha(\beta),\rho}. \label{4.1a}%
\end{equation}

(iii) If $\sigma_{0}>0$ there is family of control sets $D_{1}^{\alpha
(\beta),\rho}$ defined for $\rho>0$ and $\beta\in(-\overline{\beta},0)$, such
that (\ref{4.1a}) holds for all $\rho$ and $\beta$.
\end{theorem}

\begin{proof}
(i): Since the set $\Gamma_{0}\cup x_{0}$ is a maximal chain transitive set
and the inner pair condition (\ref{pair}) holds, Theorem \ref{Theorem_inner}
shows that there is an increasing family of control sets $D_{0}^{0,\rho}%
,\rho>0$, of (\ref{4.1}) with $\alpha_{0}=0$ such that%
\[
\Gamma_{0}\cup x_{0}\subset\mathrm{int}D_{0}^{0,\rho}\text{ and }\Gamma
_{0}\cup x_{0}=\bigcap\nolimits_{\rho>0}D_{0}^{0,\rho}.
\]
Since the accessibility rank condition (\ref{ARC}) holds for $\alpha_{0}=0$,
Theorem \ref{Theorem_parameter1} shows that for every $\rho>0$ and some
$\alpha_{0}(\rho)>0$ there is a unique lower semicontinuous family of control
sets $D_{0}^{\alpha,\rho}$ with parameters $\left\vert \alpha\right\vert
<\alpha_{0}(\rho)$ containing $\Gamma_{0}\cup x_{0}$ in the interior. With
$\beta_{0}(\rho)=\beta(\alpha_{0}(\rho))$ assertion (i) follows.

(ii) and (iii)\textit{:} By Theorem \ref{th:Andr} there is a neighborhood
$U_{0}$ of $\Gamma_{0}\cup x_{0}$ in which a unique limit cycle $L_{\beta
},\left\vert \beta\right\vert \in(0,\overline{\beta})$, bifurcates from
$\Gamma_{0}$. If $\sigma_{0}<0$ the limit cycle exists and is asymptotically
stable for $\beta>0$, and if $\sigma_{0}>0$ it exists and is unstable for
$\beta<0$. The limit cycle $L_{\beta}$ is a maximal chain transitive set for
the uncontrolled equation $\dot{x}=f_{0}(\alpha(\beta),x)$, hence Theorem
\ref{Theorem_inner} shows that for every limit cycle $L_{\beta}$ there is an
increasing family of control sets $D_{1}^{\alpha(\beta),\rho},\rho>0$, of
(\ref{4.1}) with (\ref{4.1a}).
\end{proof}

\begin{remark}
Theorem \ref{th:Andr} does not yield any information about the behavior of the
uncontrolled system outside of some neighborhood of the homoclinic orbit.
Hence $\Gamma_{0}\cup x_{0}$ may be a maximal chain transitive set only in an
isolating neighborhood. In that case, the sets $D_{0}^{\alpha(\beta),\rho}$
will only be local control sets, cf. Definition \ref{Definition_local}.
\end{remark}

\begin{remark}
Theorem \ref{Theorem4.1} shows that we may consider $(\beta,\rho)$ as the
parameters which determine the behavior of the control sets. In assertions
(i)-(iii), for $i=1,2$ the maps $\rho\mapsto D_{i}^{\alpha(\beta),\rho}$ are
increasing for every $\beta$ and by Theorem \ref{Theorem_parameter1} the maps
$\beta\mapsto D_{i}^{\alpha(\beta),\rho}$ are lower semicontinuous for every
$\rho$.
\end{remark}

\begin{remark}
\label{Remark_HS1}H\"{a}ckl and Schneider \cite{HS} consider control sets near
a Takens-Bogdanov singularity, analytically and numerically, for%
\[
\dot{x}=y,~\dot{y}=\lambda_{1}+\lambda_{2}x+x^{2}+xy+u(t),\quad u(t)\in
\lbrack-\rho,\rho].
\]
Here for all parameters $(\lambda_{1},\lambda_{2})\in\mathbb{R}^{2}$ the
bifurcation behavior of the uncontrolled equation is known. For parameters
$(\lambda_{1},\lambda_{2})$ in a subset $k_{S}\subset\mathbb{R}^{2}$ a
homoclinic bifurcation occurs and one obtains a control set containing the
homoclinic orbit, cf. \cite[Figure 4]{HS} (and an invariant control set around
the asymptotically stable focus surrounded by the homoclinic orbit). For
$(\lambda_{1},\lambda_{2})$ in $C\subset\mathbb{R}^{2}$ an unstable periodic
orbit has bifurcated from the homoclinic orbit. It is contained in a variant
control set, cf. \cite[Figure 2]{HS}.
\end{remark}

\begin{remark}
While in Theorem \ref{Theorem4.1} the homoclinic orbit vanishes for
$\beta\not =0$, the implicit function theorem implies that there are
hyperbolic equilibria $x_{\alpha(\beta)}$ for the uncontrolled system which
depend continuously on $\beta$. The local behavior near these equilibria will
play a certain role in the proof of Theorem \ref{Theorem_homoclinic}.
\end{remark}

\begin{remark}
\label{Remark_x_beta}The Index Theorem (see Wiggins \cite[Corollary
6.0.2]{S2003}) implies that inside any limit cycle $L_{\beta}$ of the
uncontrolled system there is at least one fixed point $x_{2}^{\alpha(\beta)}$.
Under the inner pair condition, one finds by Theorem \ref{Theorem_inner}
control sets with $x_{2}^{\alpha(\beta)}\in\mathrm{int}D_{2}^{\alpha
(\beta),\rho}$.
\end{remark}

Theorem \ref{Theorem4.1} does not answer the question, when the control sets
$D_{0}^{\alpha(\beta),\rho}$ and $D_{1}^{\alpha(\beta),\rho}$ coincide, in
their common range of definition. The following corollary shows, in
particular, how this equality depends on the relation between the parameters
$\beta$ and $\rho$. For simplicity we suppose that $\sigma_{0}<0$. Thus both
control sets $D_{0}^{\alpha(\beta),\rho}$ and $D_{1}^{\alpha(\beta),\rho}$
exist for $\rho>0$ and $\beta\in(0,\beta_{0}(\rho))$.

\begin{corollary}
\label{Corollary_relation}Let the assumptions of Theorem \ref{Theorem4.1} be
satisfied and assume that $\sigma_{0}<0$.

(i) For every $\beta\in(0,\overline{\beta})$ there is $\rho_{1}(\beta)>0$ such
that for all $\rho\in(0,\rho_{1}(\beta)]$ the set $D_{1}^{\alpha(\beta),\rho}$
is an invariant control set.

(ii) For every $\rho>0$ there is $\underline{\beta}(\rho)>0$ such that for all
$\beta\in(0,\underline{\beta}(\rho))$ the control sets coincide,
$D_{0}^{\alpha(\beta),\rho}=D_{1}^{\alpha(\beta),\rho}$.

(iii) If $D_{0}^{\alpha(\beta),\rho}$ is a variant control set for some
$\beta\in(0,\overline{\beta}),\rho\in(0,\rho_{1}(\beta)]$, then $D_{0}%
^{\alpha(\beta),\rho}\not =D_{1}^{\alpha(\beta),\rho}$.
\end{corollary}

\begin{proof}
(i) This follows from Theorem \ref{Theorem_inner}, since the periodic orbits
$L_{\beta}$ are asymptotically stable.

(ii) Fix $\rho>0$. Since $\Gamma_{0}\cup x_{0}\subset\mathrm{int}D_{0}%
^{0,\rho}$ and the periodic orbits $L_{\beta},\beta>0$, bifurcate from this
homoclinic orbit, it follows that there is $\beta^{\prime}(\rho)>0$ such that
$L_{\beta}\subset\mathrm{int}D_{0}^{0,\rho}$ for all $\beta\in\left(
0,\beta^{\prime}(\rho)\right]  $. Define a compact set $K\subset D_{0}%
^{0,\rho}$ by%
\[
K:=\left(  \Gamma_{0}\cup x_{0}\right)  \cup\bigcup\nolimits_{\beta\in
(0,\beta^{\prime}(\rho)]}L_{\beta}.
\]
By Theorem \ref{Theorem_parameter1} it follows that there is $\beta
^{\prime\prime}(\rho)\in(0,\beta_{0}(\rho))$ such that for all $\beta
\in(0,\beta^{\prime\prime}(\rho))$ the inclusion $K\subset\mathrm{int}%
D_{0}^{\alpha(\beta),\rho}$ holds. Thus for $\underline{\beta}(\rho
):=\min\left\{  \beta^{\prime}(\rho),\beta^{\prime\prime}(\rho)\right\}  $ it
follows that%
\[
L_{\beta}\subset\mathrm{int}D_{0}^{\alpha(\beta),\rho}\text{ and hence }%
D_{0}^{\alpha(\beta),\rho}=D_{1}^{\alpha(\beta),\rho}\text{ for all }\beta
\in(0,\underline{\beta}(\rho)).
\]

(iii) By assertion (i) $D_{1}^{\alpha(\beta),\rho}$ is an invariant control
set, hence it cannot coincide with the variant control set $D_{0}%
^{\alpha(\beta),\rho}$.
\end{proof}

\begin{remark}
\label{Remark4.8}Corollary \ref{Corollary_relation} reveals the subtle
relation between the size of the control range determined by $\rho$ and the
bifurcation parameter $\beta$. In assertion (i), the control set
$D_{1}^{\alpha(\beta),\rho}$ around the asymptotically stable periodic orbit
$L_{\beta}$ is invariant for small $\rho>0$; here one will expect $\rho
_{1}(\beta)\rightarrow0$ for $\beta\rightarrow0$. In assertion (ii), $\rho>0$
is fixed and the homoclinic orbit is contained in the interior of the control
set $D_{0}^{0,\rho}$. Since $L_{\beta}\rightarrow\Gamma_{0}\cup x_{0}$ for
$\beta\rightarrow0$, it follows that $L_{\beta}\subset D_{0}^{0,\rho}$ for
$\beta$ small enough; here $\underline{\beta}(\rho)\rightarrow0$ for
$\rho\rightarrow0$. In assertion (iii), $\beta\in(0,\overline{\beta}),\rho
\in(0,\rho_{1}(\beta)]$ is small enough to guarantee by (i) that
$D_{1}^{\alpha(\beta),\rho}$ is invariant. If $D_{0}^{\alpha(\beta),\rho}$ is
variant, this implies that the control sets cannot coincide. In view of (ii),
this can only happen if $\beta\geq\underline{\beta}(\rho)$, hence $\rho$ must
be small enough. The assumption that $D_{0}^{\alpha(\beta),\rho}$ is variant
appears to be mild, since the hyperbolic equilibrium satisfies $x_{0}%
\in\mathrm{int}D_{0}^{\alpha(\beta),\rho}$ for $\rho>0$ and $\beta\in
\lbrack0,\beta_{0}(\rho))$, and hence all points on the unstable manifold of
$x_{0}$ can be reached from $D_{0}^{\alpha(\beta),\rho}$. See the example in
Section \ref{Section5} for an illustration.
\end{remark}

The next theorem shows that the qualitative behavior of the control system can
be different from the behavior of the uncontrolled system. More precisely, we
find parameter regions where there is no homoclinic orbit and no limit cycle
for the uncontrolled system while there exist periodic orbits of the control
system\ which are arbitrarily close to the homoclinic orbit.

\begin{theorem}
\label{Theorem_homoclinic}Let the assumptions of Theorem \ref{Theorem4.1} be
satisfied and assume, additionally, that the system with $\alpha_{0}=0$
linearized in $(0,0)\in\mathbb{R}^{d}\times\mathbb{R}^{m}$ is controllable.
Then there is a neighborhood $U_{1}$ of the homoclinic orbit $\Gamma_{0}\cup
x_{0}$ such that for every $\delta>0$ there are a nonvoid parameter region
$A\subset\mathbb{R}$ and $\rho_{0}>0$ such that

(i) for $\alpha\in A$ and $\rho\in(0,\rho_{0})$ there are periodic orbits
$\varphi^{\alpha}(\cdot,y,u)\subset D_{0}^{\alpha,\rho},u\in\mathcal{U}^{\rho
}$, with Hausdorff distance $d_{H}(\varphi^{\alpha}(\cdot,y,u),\Gamma_{0}\cup
x_{0})<\delta$;

(ii) for $\alpha\in A$ the uncontrolled system $\dot{x}=f_{0}(\alpha,x)$ has
no homoclinic orbit or periodic solution in $U_{1}$ except for the hyperbolic
equilibrium $x_{\alpha}$.
\end{theorem}

\begin{proof}
Recall that the hyperbolic equilibrium $x_{0}$ yields for $\alpha$ near $0$
hyperbolic equilibria $x_{\alpha}$ which depend continuously on $\alpha$.
Suppose first that $\sigma_{0}<0$. Theorem \ref{th:Andr} shows that in a
neighborhood $U_{0}$ of $\Gamma_{0}\cup x_{0}$ for $\left\vert \alpha
\right\vert $ small enough a unique limit cycle $L_{\beta(\alpha)}$ bifurcates
from $\Gamma_{0}$. It exists if and only if $\beta(\alpha)>0$. We will
consider $\alpha$ with $\beta(\alpha)<0$, hence the uncontrolled equation
$\dot{x}=f_{0}(\alpha,x)$ has no periodic solution except for the equilibrium
$x_{\alpha}$ and assertion (ii) holds. For the proof of (i) it is convenient
to suppose that $\alpha(\beta)>0$ for $\beta<0$ (otherwise, we replace
$\alpha$ by $-\alpha$).

By Theorem \ref{Theorem4.1}(i) with $\beta=0$ there is for $\rho>0$ a control
set $D_{0}^{0,\rho}$ satisfying%
\[
\Gamma_{0}\cup x_{0}\subset\mathrm{int}D_{0}^{0,\rho}\text{ and }\Gamma
_{0}\cup x_{0}=\bigcap\nolimits_{\rho>0}D_{0}^{0,\rho}.
\]
By Theorem \ref{Theorem3.6} there are unique local control sets $D_{loc}%
^{0,\rho}$ such that the equilibrium $x_{0}$ of the uncontrolled system
satisfies
\[
x_{0}\in\mathrm{int}D_{loc}^{0,\rho}\text{ and }\{x_{0}\}=\bigcap
\nolimits_{\rho>0}D_{loc}^{0,\rho}.
\]
We can choose $\rho>0$ so small that%
\begin{equation}
\sup\left\{  \left\Vert z-x_{0}\right\Vert \left\vert z\in D_{loc}^{0,\rho
}\right.  \right\}  <\delta. \label{Haus1}%
\end{equation}
Since $D_{loc}^{0,\rho}\cap D_{0}^{0,\rho}\not =\varnothing$, it follows that
$D_{loc}^{0,\rho}\subset D_{0}^{0,\rho}$ and for $\delta>0$, small enough,
$D_{loc}^{0,\rho}\not =D_{0}^{0,\rho}$.

Let $y\in\Gamma_{0}\cap\mathrm{int}D_{loc}^{0,\rho}$. Since $\varphi
^{0}(t,y,0)\rightarrow x_{0}\in\mathrm{int}D_{loc}^{0,\rho}$ for
$t\rightarrow\pm\infty$, there is $T>0$ such that the homoclinic trajectory
satisfies $\varphi^{0}(t,y,0)\in\mathrm{int}D_{loc}^{0,\rho}$ for all
$\left\vert t\right\vert \geq T$ and $\varphi^{0}(\tau,y,0)$ is not in the
isolating neighborhood of $D_{loc}^{0,\rho}$ for some $\tau\in(0,T)$. By
continuous dependence of the solution on the parameter $\alpha$, there is
$\alpha_{0}(\rho)>0$ such that $\varphi^{\alpha}(T,y,0)\in\mathrm{int}%
D_{loc}^{0,\rho}$ for all $\alpha\in(0,\alpha_{0}(\rho)]$. Choose $\alpha
_{0}(\rho)$ small enough such that the Hausdorff distance
\begin{equation}
d_{H}\left(  \{\varphi^{\alpha}(t,y,0)\left\vert t\in\lbrack0,T]\right.
\},\Gamma_{0}\cup x_{0}\right)  <\delta\text{ for }\alpha\in(0,\alpha_{0}%
(\rho)], \label{Haus2}%
\end{equation}
where we use $\left\{  \varphi^{0}(t,y,0)\left\vert t\in\mathbb{R}\right.
\right\}  =\Gamma_{0}$. The compact set%
\[
K:=\left\{  y\right\}  \cup\left\{  \varphi^{\alpha}(T,y,0)\left\vert
\alpha\in\lbrack0,\alpha_{0}(\rho)]\right.  \right\}
\]
is contained in $\mathrm{int}D_{loc}^{0,\rho}$. Theorem
\ref{Theorem_parameter1} applied to local control sets implies that there is
$\alpha_{1}(\rho)\in(0,\alpha_{0}(\rho)]$ such that for all $\alpha\in
\lbrack0,\alpha_{1}(\rho)]$%
\[
\left\{  y\right\}  \cup\left\{  \varphi^{\alpha}(T,y,0)\left\vert \alpha
\in\lbrack0,\alpha_{1}(\rho)]\right.  \right\}  \subset K\subset
\mathrm{int}D_{loc}^{\alpha,\rho}.
\]
There are a control $u^{0}\in\mathcal{U}^{\rho}$ and a time $T^{0}>0$ such
that $\varphi^{0}(T^{0},\varphi^{\alpha}(T,y,0),u^{0})=y$. Then Remark
\ref{Remark_near} implies that one may choose $\alpha_{2}(\rho)\in
(0,\alpha_{1}(\rho)]$ such that for all $\alpha\in(0,\alpha_{2}(\rho))$ there
are $u^{\alpha}\in\mathcal{U}^{\rho}$ and $T^{\alpha}>0$ with $\varphi
^{\alpha}(T^{\alpha},\varphi^{\alpha}(T,y,0),u^{\alpha})=y$ and trajectories
$\varphi^{\alpha}(t,\varphi^{\alpha}(T,y,0),u^{\alpha}),\,t\in\lbrack
0,T^{\alpha}]$, arbitrarily close to $\varphi^{0}(t,\varphi^{\alpha
}(T,y,0),u^{0}),t\in\lbrack0,T^{0}]$, hence contained in $D_{loc}^{0,\rho}$.
By (\ref{Haus1}) and (\ref{Haus2}) this implies that the Hausdorff distance of
the resulting controlled periodic orbit to $\Gamma_{0}\cup x_{0}$ is smaller
than $\delta$, hence assertion (i) holds in the case $\sigma_{0}<0$..

For $\sigma_{0}>0$ consider $\alpha$ with $\beta(\alpha)>0$, where the
uncontrolled equation $\dot{x}=f_{0}(\alpha,x)$ has no periodic solution
except for the equilibrium $x_{\alpha}$. Then the assertion is proved analogously.
\end{proof}

\begin{remark}
\label{Remark_HS2}In their analysis of the Takens-Bogdanov equation, H\"{a}ckl
and Schneider \cite[Theorem 4.7]{HS} prove that there exist parameter values
and control ranges such that the control system has an at least doubly
connected control set while for all constant controls only equilibrium points
exist as limit sets. Here they use that the control directly affects the
bifurcation parameter $\lambda_{1}$.
\end{remark}

\subsection{The three-dimensional case}

The following theorems analyze the control sets in $\mathbb{R}^{3}$ when the
uncontrolled system undergoes a homoclinic bifurcation in the situation of
Theorem \ref{Theorem2.2} and Theorem \ref{Theorem2.3}.

If the uncontrolled system $\dot{x}=f_{0}(0,x)$ satisfies the hypotheses of
Theorem \ref{Theorem2.2}(i), it has an orbit $\Gamma_{0}$ homoclinic to a
saddle equilibrium point $x_{0}=0$, and $\dot{x}=f_{0}(\alpha,x)$ undergoes a
homoclinic bifurcation with saddle quantity $\sigma_{0}=\lambda_{1}%
(0)+\lambda_{2}(0)<0$, a split function $\beta(\alpha)$ with $\beta^{\prime
}(0)\not =0$, and bifurcating unique and asymptotically stable limit cycles
$L_{\beta}$ defined for $0<\left\vert \beta\right\vert <\overline{\beta}$ and
we may write $\alpha=\alpha(\beta)$. We use the notation from Theorem
\ref{Theorem2.2}.

\begin{theorem}
\label{Theorem4.5}Consider a family of control-affine systems in
$\mathbb{R}^{3}$ of the form (\ref{4.1}) and suppose that the accessibility
rank condition (\ref{ARC}) holds for $\alpha_{0}=0$ and that the control
system satisfies the inner pair condition (\ref{pair}) for all $x\in
\mathbb{R}^{3}$. Assume that the uncontrolled system $\dot{x}=f_{0}(0,x)$ has
an orbit $\Gamma_{0}$ homoclinic to a saddle $x_{0}=0$ with real eigenvalues
$\lambda_{1}(0)>0>\lambda_{2}(0)>\lambda_{3}(0)$, that $\Gamma_{0}\cup x_{0}$
is a maximal chain transitive set and the assumptions of Theorem
\ref{Theorem2.2}(i) are satisfied.

(i) Then there is a family of control sets $D_{0}^{\alpha(\beta),\rho}$,
defined for $\rho>0$ and $\beta\in\left(  -\beta_{0}(\rho),\beta_{0}%
(\rho)\right)  $ with $\beta_{0}(\rho)\in(0,\overline{\beta})$, satisfying for
all $\rho$ and $\beta$%
\begin{equation}
\Gamma_{0}\cup x_{0}\subset\mathrm{int}D_{0}^{\alpha(\beta),\rho}\text{ and
}\Gamma_{0}\cup x_{0}=\bigcap\nolimits_{\rho>0}D_{0}^{0,\rho}. \label{3d.1}%
\end{equation}

(ii) There is a family of control sets $D_{1}^{\alpha(\beta),\rho}$, defined
for $\rho>0$ and $\beta\in(0,\overline{\beta})$, satisfying for all $\rho$ and
$\beta$%
\begin{equation}
L_{\beta}\subset\mathrm{int}D_{1}^{\alpha(\beta),\rho}\text{ and }L_{\beta
}=\bigcap\nolimits_{\rho>0}D_{1}^{\alpha(\beta),\rho}. \label{3d.2}%
\end{equation}
Furthermore, for every $\beta\in(0,\overline{\beta})$ there is $\rho_{1}%
(\beta)$ such that for every $\rho\in(0,\rho_{1}(\beta)]$ the set
$D_{1}^{\alpha(\beta),\rho}$ is an invariant control set.
\end{theorem}

\begin{proof}
(i) The set $\Gamma_{0}\cup x_{0}$ is a maximal chain transitive set for the
uncontrolled equation $\dot{x}=f_{0}(0,x)$. Theorem \ref{Theorem_inner} shows
that there is an increasing family of control sets $D_{0}^{0,\rho},\rho>0$, of
(\ref{4.1}) with $\alpha_{0}=0$ such that%
\[
\Gamma_{0}\cup x_{0}\subset\mathrm{int}D_{0}^{0,\rho}\text{ and }\Gamma
_{0}\cup x_{0}=\bigcap\nolimits_{\rho>0}D_{0}^{0,\rho}.
\]
Theorem \ref{Theorem_parameter1} shows that for every $\rho>0$ and some
$\alpha_{0}(\rho)>0$ there is a unique lower semicontinuous family of control
sets $D_{0}^{\alpha,\rho}$ with parameters $\left\vert \alpha\right\vert
<\alpha_{0}(\rho)$ containing $\Gamma_{0}\cup x_{0}$ in the interior. With
$\beta_{0}(\rho)=\beta(\alpha_{0}(\rho))$ assertion (i) follows.

(ii) By Theorem \ref{Theorem2.2}(i) there is a neighborhood $U_{0}$ of
$\Gamma_{0}\cup x_{0}$ in which a unique and asymptotically stable limit cycle
$L_{\beta},\beta\in(0,\overline{\beta})$, bifurcates from $\Gamma_{0}$. The
limit cycle $L_{\beta}$ is a maximal chain transitive set for the uncontrolled
equation $\dot{x}=f_{0}(\alpha(\beta),x)$, hence Theorem \ref{Theorem_inner}
shows that for every limit cycle $L_{\beta}$ there is an increasing family of
control sets $D_{1}^{\alpha(\beta),\rho},\rho>0$, of (\ref{4.1}) with%
\[
L_{\beta}\subset\mathrm{int}D_{1}^{\alpha(\beta),\rho}\text{ and }L_{\beta
}=\bigcap\nolimits_{\rho>0}D_{1}^{\alpha(\beta),\rho}.
\]
Theorem \ref{Theorem_inner} shows that the control sets $D_{1}^{\alpha
(\beta),\rho}$ containing the asymptotically stable limit cycle $L_{\beta}$
are invariant for $\rho>0$, small enough.
\end{proof}

Similarly one obtains the following result if the assumptions of Theorem
\ref{Theorem2.2}(ii), in particular, $\sigma_{0}>0$, are satisfied, and hence
a unique saddle limit cycle $L_{\beta}$ bifurcates from the homoclinic orbit
$\Gamma_{0}$ in a neighborhood $U_{0}$ of $\Gamma_{0}\cup x_{0}$.

\begin{theorem}
\label{Theorem4.6}In the situation of Theorem \ref{Theorem4.5} suppose that
the uncontrolled system $\dot{x}=f_{0}(0,x)$ satisfies the assumptions of
Theorem \ref{Theorem2.2}(ii).

(i) Then there is a family of control sets $D_{0}^{\alpha(\beta),\rho}$
defined for $\rho>0$ and $\beta\in\left(  -\beta_{0}(\rho),\beta_{0}%
(\rho)\right)  $ with $\beta_{0}(\rho)\in(0,\overline{\beta})$ such that
(\ref{3d.1}) holds for all $\rho$ and $\beta$.

(ii) If $\Gamma_{0}$ is simple, there is a family of control sets
$D_{1}^{\alpha(\beta),\rho}$ defined for $\rho>0$ and $\beta\in(-\overline
{\beta},0)$ such that (\ref{3d.2}) holds for all $\rho$ and $\beta$.

(iii) If $\Gamma_{0}$ is twisted, there is a family of control sets
$D_{1}^{\alpha(\beta),\rho}$ defined for $\rho>0$ and $\beta\in(0,\overline
{\beta})$ such that (\ref{3d.2}) holds for all $\rho$ and $\beta$.
\end{theorem}

\begin{proof}
The proof of this theorem follows the same steps as the proof of Theorem
\ref{Theorem4.5}. One has to use that for simple $\Gamma_{0}$ the bifurcating
limit cycles $L_{\beta}$ exist for $\beta<0$ and for twisted $\Gamma_{0}$ they
exist for $\beta>0$.
\end{proof}

The remarkable result here is that the direction of bifurcation for the
control sets $D_{1}^{\alpha(\beta),\rho}$ depends on topological property if
the stable manifold $W^{s}$ of $x_{0}$ is simple or twisted.

Finally, we obtain the following result for a homoclinic bifurcation of a
saddle-focus with $\sigma_{0}<0$ as described in Theorem \ref{Theorem2.3}.

\begin{theorem}
\label{Theorem4.7}Consider a family of control-affine systems in
$\mathbb{R}^{3}$ of the form (\ref{4.1}) and suppose that the accessibility
rank condition (\ref{ARC}) holds for $\alpha_{0}=0$ and that the control
system satisfies the inner pair condition (\ref{pair}) for all $x\in
\mathbb{R}^{3}$. Assume that the uncontrolled system $\dot{x}=f_{0}(0,x)$ has
an orbit $\Gamma_{0}$ homoclinic to a saddle-focus $x_{0}=0$ satisfying the
assumptions of Theorem \ref{Theorem2.3}, and $\Gamma_{0}\cup x_{0}$ is a
maximal chain transitive set.

(i) Then there is a family of control sets $D_{0}^{\alpha(\beta),\rho}$
defined for $\rho>0$ and $\beta\in\left(  -\beta_{0}(\rho),\beta_{0}%
(\rho)\right)  $ with $\beta_{0}(\rho)\in(0,\overline{\beta})$ such that
(\ref{3d.1}) holds for all $\rho$ and $\beta$.

(ii) There is a family of control sets $D_{1}^{\alpha(\beta),\rho}$ defined
for $\rho>0$ and $\beta\in(0,\overline{\beta})$ such that (\ref{3d.2}) holds
for all $\rho$ and $\beta$. Furthermore, for every $\beta\in(0,\overline
{\beta})$ there is $\rho_{1}(\beta)$ such that for every $\rho\in(0,\rho
_{1}(\beta)]$ the set $D_{1}^{\alpha(\beta),\rho}$ is an invariant control set.
\end{theorem}

\begin{proof}
The proof of this theorem follows the same steps as the proof of Theorem
\ref{Theorem4.5}.One also has to use that the bifurcating limit cycles
$L_{\beta},\beta<0$, are asymptotically stable, hence in assertion (ii) one
obtains invariant control sets.
\end{proof}

\begin{remark}
In all situations analyzed in Theorems \ref{Theorem4.5}, Theorem
\ref{Theorem4.6}, and Theorem \ref{Theorem4.7}, one can obtain results
analogous to Corollary \ref{Corollary_relation}, and to Theorem
\ref{Theorem_homoclinic} on the existence of controlled homoclinic orbits in
parameter regions where no periodic solutions exist for the uncontrolled
system. This holds, since the corresponding proofs do not use that the
dimension of the state space is two.
\end{remark}

\section{An example\label{Section5}}

Consider the following planar control system%
\begin{align}
\dot{x}  &  =-x+2y+x^{2}\label{5.1}\\
\dot{y}  &  =(2-\alpha)x-y-3x^{2}+\frac{3}{2}xy+u(t)\nonumber
\end{align}
with $u(t)\in U=[-\rho,\rho],\,\rho>0$. This is a special case of (\ref{4.1})
with
\[
f_{0}(\alpha,x,y)=\left[
\begin{array}
[c]{c}%
-x+2y+x^{2}\\
(2-\alpha)x-y-3x^{2}+\frac{3}{2}xy
\end{array}
\right]  ,\quad f_{1}(x,y)=\left[
\begin{array}
[c]{c}%
0\\
1
\end{array}
\right]  .
\]
For $u=0$ one obtains Sandstede's example of a homoclinic bifurcation, cf.
Sandstede \cite{Sand}, Kuznetsov \cite[Example 6.1]{Y1998}. For an application
of Theorem \ref{Theorem4.1}, we first check that this uncontrolled system
suffers a homoclinic bifurcation according to Theorem \ref{th:Andr}.

The origin $(x_{0},y_{0})=(0,0)$ is an equilibrium for all $\alpha$ and it is
a saddle for sufficiently small $\left\vert \alpha\right\vert $. For
$\alpha_{0}=0$ one obtains $\sigma_{0}=\lambda_{1}(0)+\lambda_{2}%
(0)=1-3=-2<0$. One can show that there is a homoclinic orbit contained in the
set of all $(x,y)$ with%
\[
x^{2}(1-x)-y^{2}=0,
\]
hence $y=\pm x\sqrt{1-x}$ for all points $(x,y)$ on the homoclinic orbit. As
noted above, the condition $\beta^{\prime}(0)\not =0$ is equivalent to the
Melnikov condition (\ref{M6.25}), which here has the form%
\[
M_{\alpha_{0}}(0)=-\int_{-\infty}^{\infty}\exp\left[  -\int_{0}^{t}\left(
-2+\frac{7}{2}x\right)  d\tau\right]  x\dot{x}dt.
\]
Write the first component of the homoclinic trajectory for $y>0$ as
$x(t)=x^{+}(t)$ and for $y<0$ as $x(t)=x^{-}(t)$. Thus the equation for
$x(\cdot)$ can be written as%
\[
\dot{x}^{+}=x(x-1+2\sqrt{1-x})>0,~\dot{x}^{-}=-x(x-1-2\sqrt{1-x})<0.
\]
Observe that $(x_{1},y_{1})=(1,0)$ is on the homoclinic orbit and we may
suppose that the homoclinic solution satisfies $x(0)=x_{1}=1$. Define
$h(t)=\exp\left[  -\int_{0}^{t}\left(  -2+\frac{7}{2}x\right)  d\tau\right]
,\allowbreak t\in\mathbb{R}$. Then the integral for $M_{\alpha}(0)$ can be
written as%
\[
\int_{-\infty}^{0}h(t)x\dot{x}dt+\int_{0}^{\infty}h(t)x\dot{x}dt=\int_{0}%
^{1}h(t^{+}(x^{+}))x^{+}dx^{+}+\int_{0}^{1}h(t^{-}(x^{-}))x^{-}dx^{-}>0,
\]
showing that $M_{\alpha}(0)\not =0$. Hence Theorem \ref{th:Andr} implies that
the uncontrolled equation has an asymptotically stable limit cycle for
$\alpha>0$.

Next we check the assumptions of Theorem \ref{Theorem4.1}. One computes for
$\alpha_{0}=0$%
\[
ad_{f_{0}}f_{1}(x,y)=[f_{0},f_{1}](x,y)=-\left[
\begin{array}
[c]{cc}%
\frac{\partial f_{01}}{\partial x} & \frac{\partial f_{01}}{\partial y}\\
\frac{\partial f_{02}}{\partial x} & \frac{\partial f_{02}}{\partial y}%
\end{array}
\right]  \left[
\begin{array}
[c]{c}%
0\\
1
\end{array}
\right]  =-\left[
\begin{array}
[c]{c}%
2\\
-1+\frac{3}{2}x
\end{array}
\right]  ,
\]
where $f_{0}=(f_{01},f_{02})^{\top}$. One finds that $f_{1}(x,y)=(0,1)^{\top}$
and $ad_{f_{0}}f_{1}(x,y)$ are linearly independent for all $\alpha$. Thus
condition (\ref{NamA}) holds implying the inner pair condition (\ref{pair})
and the accessibility rank condition (\ref{ARC}) for all $(x,y)\in
\mathbb{R}^{2}$. Furthermore, also the controllability condition in Theorem
\ref{Theorem_homoclinic} holds, since the control system with $\alpha_{0}=0$
linearized in $(x,y)=(0,0),u=0$ is controllable,
\[
A:=\left[  \frac{\partial f_{0}(\alpha_{0},0,0)}{\partial x},\frac{\partial
f_{0}(\alpha_{0},0,0)}{\partial y}\right]  =\left[
\begin{array}
[c]{cc}%
-1 & 2\\
2 & -1
\end{array}
\right]  ,\quad B:=f_{1}(0,0)=\left[
\begin{array}
[c]{c}%
0\\
1
\end{array}
\right]  ,
\]
hence $\mathrm{rank}[B,AB]=\mathrm{rank}\left[
\begin{array}
[c]{cc}%
0 & 2\\
1 & -1
\end{array}
\right]  =2$.

Theorem \ref{Theorem4.1} implies that for $\rho>0$ there is $\beta_{0}%
(\rho)>0$ such that there is a family of control sets $D_{0}^{\alpha
(\beta),\rho}$ satisfying for all $\rho>0$ and $\beta\in(-\beta_{0}%
(\rho),\beta_{0}(\rho))$ assertion (\ref{4.1b}). Since $\sigma_{0}<0$ it also
follows that there is a family of control sets $D_{1}^{\alpha(\beta),\rho}$
such that for all $\rho>0$ and $\beta\in(0,\overline{\beta})$ assertion
(\ref{4.1a}) holds. Corollary \ref{Corollary_relation} shows that for every
$\beta\in(0,\overline{\beta})$ there is $\rho_{1}(\beta)$ such that
$D_{1}^{\alpha(\beta),\rho}$ is an invariant control set for $\rho\in
(0,\rho_{1}(\beta)]$. For $\alpha_{0}=0,u=0$ one finds the unique equilibrium
$(\bar{x},\bar{y})=(\frac{2}{3},\frac{1}{9})$ in the interior of the region
bounded by the homoclinic orbit, cf. Remark \ref{Remark_x_beta}. It is an
unstable focus, which for $\rho>0$ small enough is contained in the interior
of an open control set, since one can check the inner pair condition
(\ref{pair}).

Figures 1 -- 5 indicate phase portraits of the uncontrolled systems and
present numerical approximations of the control sets. These results are based
on H\"{a}ckl's algorithm for the computation of reachable and controllable
sets, cf. H\"{a}ckl \cite{G1992}, Colonius and Kliemann \cite[Appendix
C]{ColK00}. A reachable set is the union of all solutions from an initial
state corresponding to admissible control functions. The solutions are
approximated via discrete-time systems (obtained here by a Runge-Kutta method
RK$5(4)$). A space discretization via a grid (here of size $150\times150$
cells) allows us to keep track of those cells that have already been reached
by some computed solution. The controllable sets are obtained via time
reversal. The implementation of H\"{a}ckl's algorithm is based on MATLAB.

For the parameter values $\alpha_{0}=0$ and $\rho=0.01$, Figure 1 shows
approximations of the control sets $D_{0}^{\alpha_{0},\rho}$ around the
homoclinic orbit $\Gamma_{0}$ and $D_{2}^{\alpha_{0},\rho}$ around the
unstable focus. The control set $D_{0}^{\alpha_{0},\rho}$ is obtained by the
intersection of the reachable and controllable sets,%
\[
D_{0}^{\alpha_{0},\rho}=\mathrm{cl}\mathcal{O}^{+}(x_{1},y_{1})\cap
\mathcal{O}^{-}(x_{2},y_{2})\text{ for any }(x_{1},y_{1}),(x_{2},y_{2}%
)\in\Gamma_{0}\subset\mathrm{int}D_{1}^{\alpha_{0},\rho}\text{. }%
\]
Hence an approximation of $D_{0}^{\alpha,\rho}$ is obtained by the
intersection of numerical approximations for the reachable and controllable
sets. The control set $D_{2}^{\alpha_{0},\rho}$ around the unstable focus
$(\bar{x},\bar{y})$ is computed as the controllable set $\mathcal{O}^{-}%
(\bar{x},\bar{y})$.

For $\alpha=-0.017241,\,\rho=0.01$, Figure 2 shows the control set
$D_{0}^{\alpha,\rho}$ around $\Gamma_{0}$ and the control set $D_{2}%
^{\alpha,\rho}$ around the unstable focus. For this negative $\alpha$-value,
no periodic orbit has bifurcated from the homoclinic orbit of the uncontrolled
equation. This illustrates Theorem \ref{Theorem_homoclinic}. For
$\alpha=0.01,\,\rho=0.01$, Figure 3\textbf{ }shows the control set
$D_{2}^{\alpha,\rho}$ around the unstable focus, and the control set
$D_{0}^{\alpha,\rho}=D_{1}^{\alpha,\rho}$ containing the homoclinic orbit
$\Gamma_{0}$ and the periodic orbit $L_{\beta(\alpha)}$, cf. Corollary
\ref{Corollary_relation}(ii). For $\alpha=0.03,\,\rho=0.01$ Figure 4 shows the
control set $D_{2}^{\alpha,\rho}$ around the unstable focus, the control set
$D_{0}^{\alpha,\rho}$ containing $\Gamma_{0}$, and the invariant control set
$D_{1}^{\alpha,\rho}$ containing the stable periodic orbit $L_{\beta(\alpha)}$
computed as the reachable set. This illustrates Corollary
\ref{Corollary_relation}(iii). Finally, for $\alpha=0.07,\,\rho=0.01$, Figure
5 shows the invariant control set $D_{1}^{\alpha,\rho}$ around the stable
periodic orbit $L_{\beta(\alpha)}$, the control set $D_{2}^{\alpha,\rho}$
around the unstable focus, and the control set $D_{0}^{\alpha,\rho}$ which has
collapsed to a control set around the saddle close to the local control sets
$D_{loc}^{0,\rho}$ used in the proof of Theorem \ref{Theorem_homoclinic}.

\section{Conclusions and open problems\label{Section6}}

Our results show, in particular, that sometimes a homoclinic bifurcation may
lead to invariant control sets (for $d=2$ this is the case in Corollary
\ref{Corollary_relation} and for $d=3$ in Theorem \ref{Theorem4.5} and Theorem
\ref{Theorem4.7}). Invariant control sets are also of interest beyond control
and deterministic perturbations, since they are the supports of invariant
densities for associated Markov diffusion processes (cf. Kliemann
\cite{Kli87}). We did not include the case of a bifurcation in $\mathbb{R}%
^{3}$ for an orbit homoclinic to a saddle-focus with saddle quantity
$\sigma_{0}>0$. This bifurcation results in an infinite number of saddle limit
cycles, cf. Kuznetsov \cite[Theorem 6.6]{Y1998}. It certainly would be of
great interest to study the controllability properties in this situation and
also for general system in $\mathbb{R}^{d}$.

\newpage%

\[%
%TCIMACRO{\FRAME{itbpFU}{11.3177cm}{7.5325cm}{0cm}{\Qcb{Fig. 1: Phase portrait
%and control sets $D_0$ around the homoclinic orbit $\Gamma_0$ and $D_2$ around
%the unstable focus for $\alpha_0=0.0$ and $\rho=0.01$}}{}{figure1.eps}%
%width 11.3177cm;  height 7.5325cm;  depth 0cm;  original-width 0pt;
%original-height 0pt;  cropleft "0";  croptop "1";  cropright "1";
%cropbottom "0";  filename 'Figure1.eps';file-properties "XNPEU";}} }%
%BeginExpansion
\raisebox{-0cm}{\parbox[b]{11.3177cm}{\begin{center}
\includegraphics[
height=7.5325cm,
width=11.3177cm
]%
{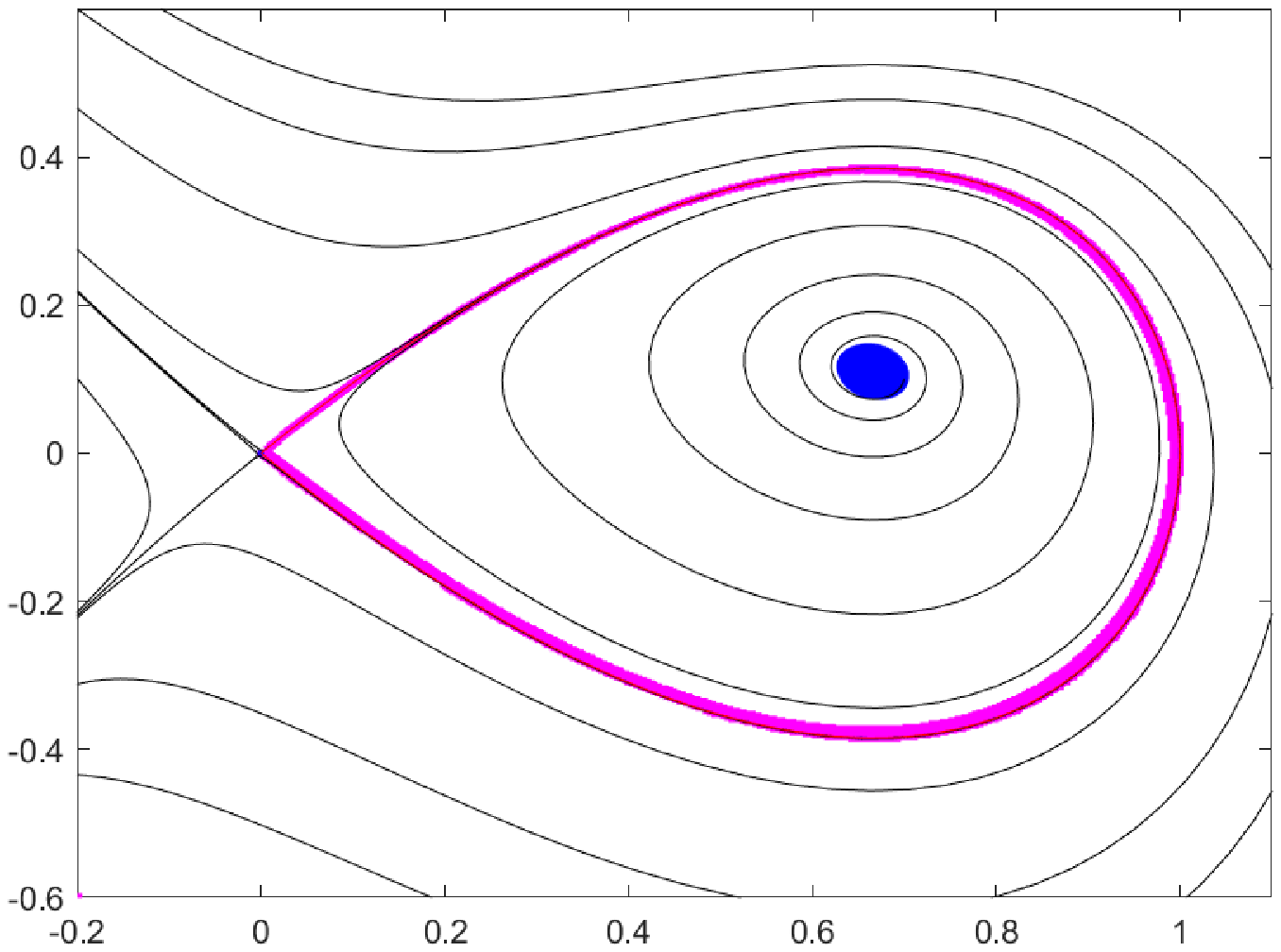}%
\\
Fig. 1: Phase portrait and control sets $D_0$ around the homoclinic orbit
$\Gamma_0$ and $D_2$ around the unstable focus for $\alpha_0=0.0$ and $%
\rho=0.01$
\end{center}}}
%EndExpansion
\]

\[%
%TCIMACRO{\FRAME{itbpFU}{11.2439cm}{7.4839cm}{0cm}{\Qcb{Fig. 2: Phase portrait
%and control sets $D_0$ around the homoclinic orbit $\Gamma_0$ and $D_2$ around
%the unstable focus for $\alpha=-0.017241,~\rho=0.01$}}{}{figure2.eps}%
%width 11.2439cm;  height 7.4839cm;  depth 0cm;  original-width 0pt;
%original-height 0pt;  cropleft "0";  croptop "1";  cropright "1";
%cropbottom "0";  filename 'Figure2.eps';file-properties "XNPEU";}} }%
%BeginExpansion
\raisebox{-0cm}{\parbox[b]{11.2439cm}{\begin{center}
\includegraphics[
height=7.4839cm,
width=11.2439cm
]%
{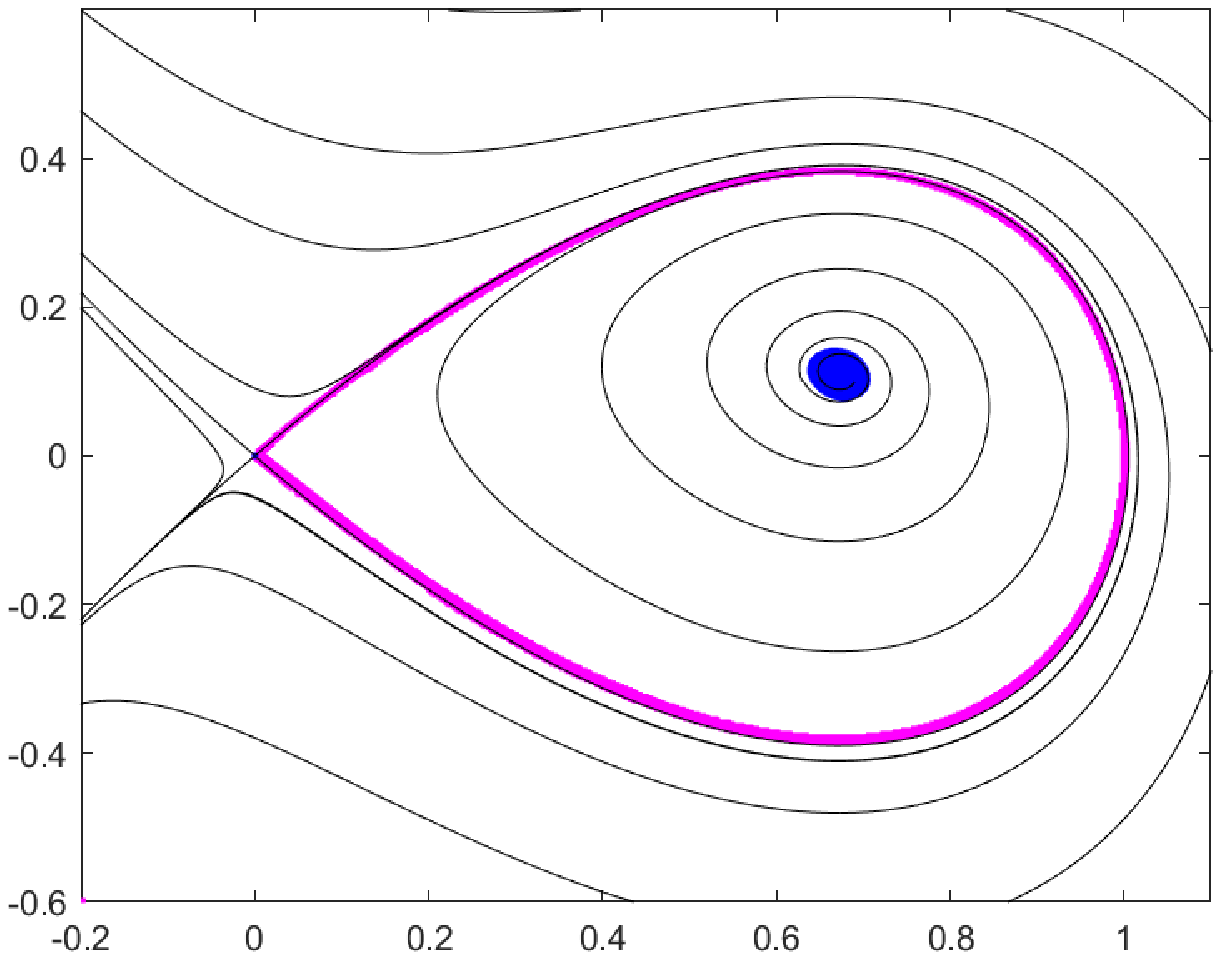}%
\\
Fig. 2: Phase portrait and control sets $D_0$ around the homoclinic orbit
$\Gamma_0$ and $D_2$ around the unstable focus for $\alpha=-0.017241,~\rho
=0.01$
\end{center}}}
%EndExpansion
\]

\[%
%TCIMACRO{\FRAME{itbpFU}{11.2439cm}{7.4839cm}{0cm}{\Qcb{Fig. 3: Phase portrait
%and control sets $D_0=D_1$ around $\Gamma_0$ and the periodic orbit, and $D_2$
%around the unstable focus for $\alpha=0.01,~\rho=0.01$}}{}{figure3.eps}%
%width 11.2439cm;  height 7.4839cm;  depth 0cm;  original-width 0pt;
%original-height 0pt;  cropleft "0";  croptop "1";  cropright "1";
%cropbottom "0";  filename 'Figure3.eps';file-properties "XNPEU";}} }%
%BeginExpansion
\raisebox{-0cm}{\parbox[b]{11.2439cm}{\begin{center}
\includegraphics[
height=7.4839cm,
width=11.2439cm
]%
{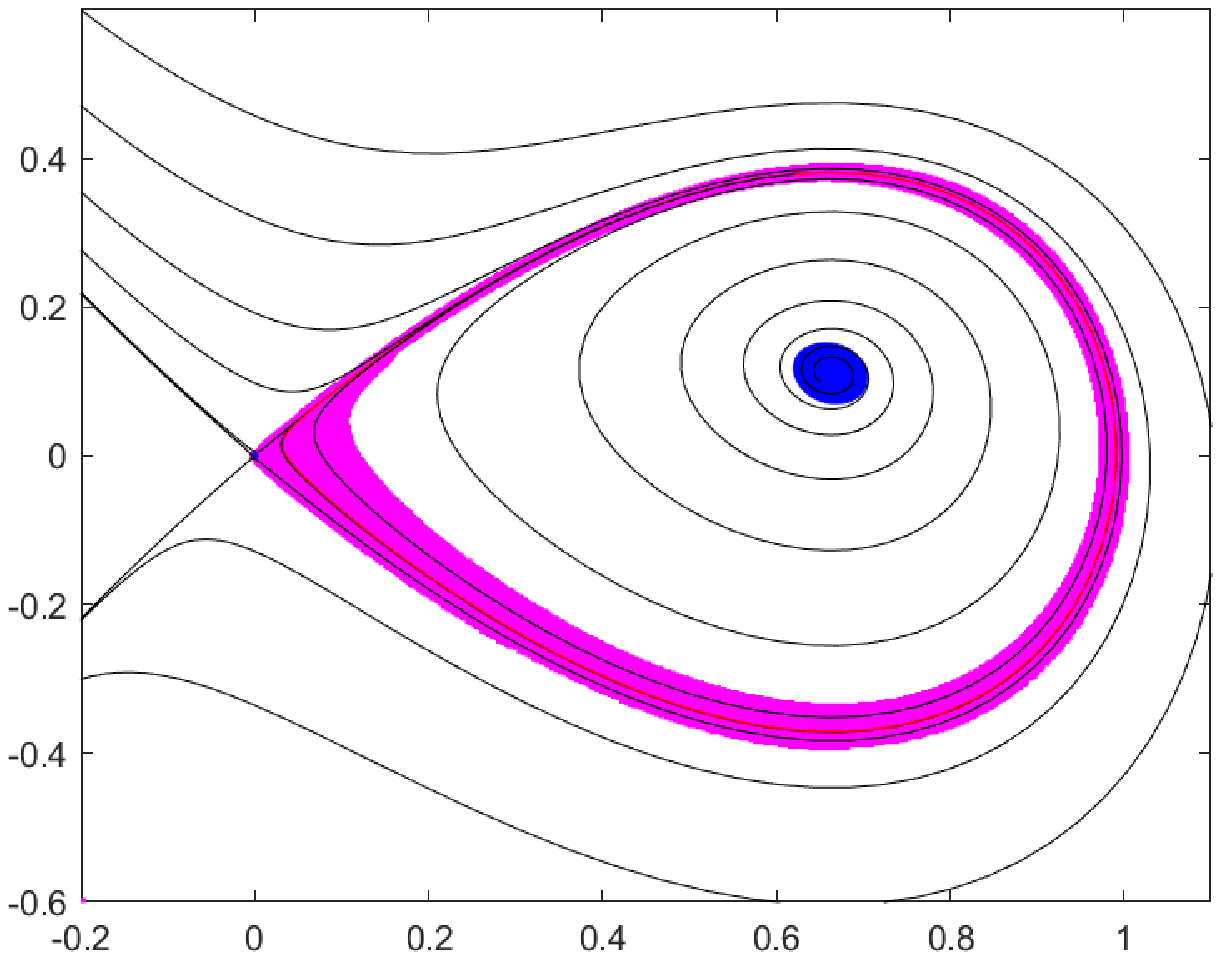}%
\\
Fig. 3: Phase portrait and control sets $D_0=D_1$ around $\Gamma_0$ and the
periodic orbit, and $D_2$ around the unstable focus for $\alpha=0.01,~\rho
=0.01$
\end{center}}}
%EndExpansion
\]

\[%
%TCIMACRO{\FRAME{itbpFU}{11.2439cm}{7.4839cm}{0cm}{\Qcb{Fig. 4: Phase portrait
%and control sets $D_0$ around $\Gamma_0$, $D_1$ around the periodic orbit, and
%$D_2$ around the unstable focus for $\alpha=0.03,~\rho=0.01$}}{}%
%{figure4.eps}{\special{ language "Scientific Word";  type "GRAPHIC";
%width 11.2439cm;  height 7.4839cm;  depth 0cm;  original-width 0pt;
%original-height 0pt;  cropleft "0";  croptop "1";  cropright "1";
%cropbottom "0";  filename 'Figure4.eps';file-properties "XNPEU";}} }%
%BeginExpansion
\raisebox{-0cm}{\parbox[b]{11.2439cm}{\begin{center}
\includegraphics[
height=7.4839cm,
width=11.2439cm
]%
{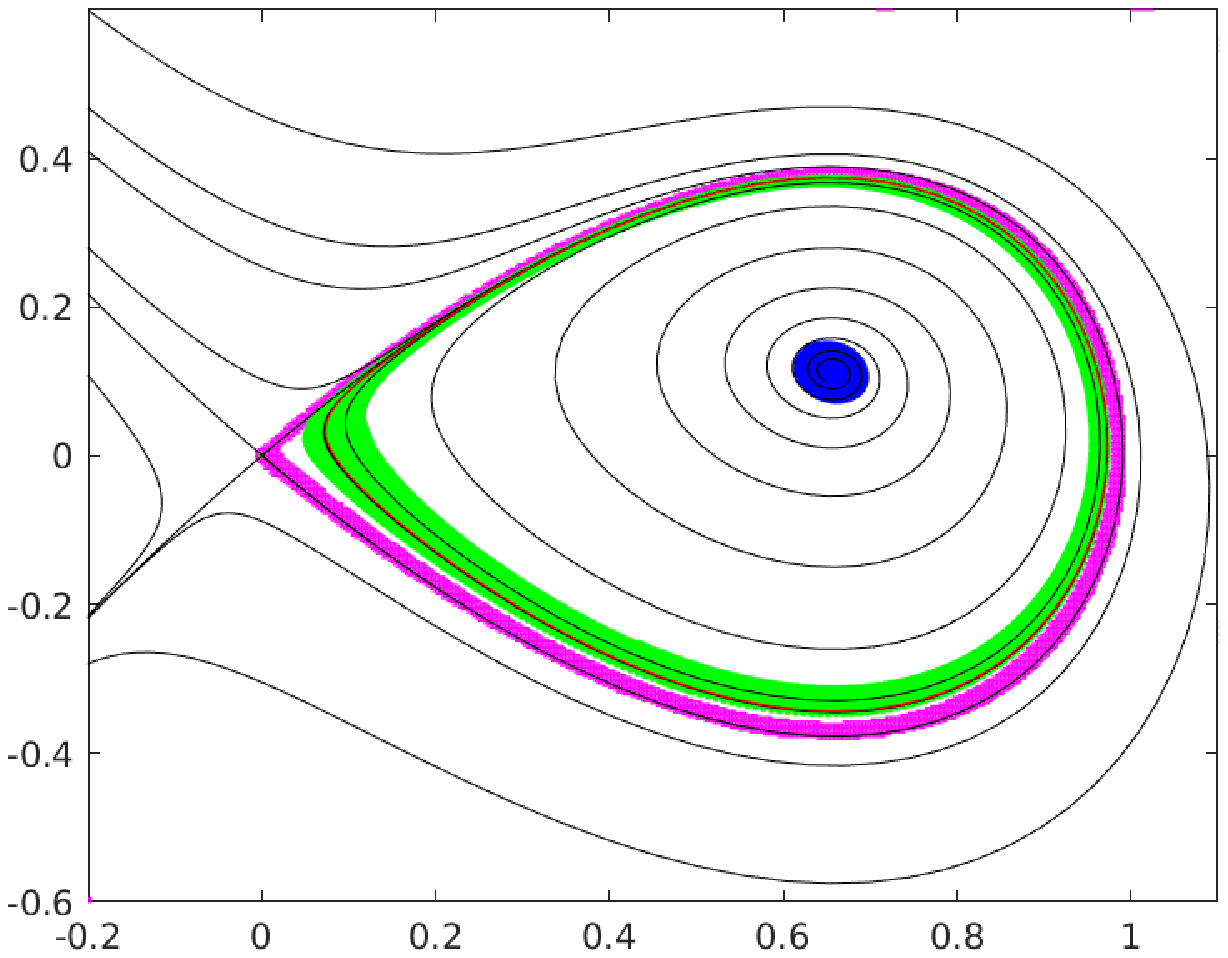}%
\\
Fig. 4: Phase portrait and control sets $D_0$ around $\Gamma_0$, $D_1$ around
the periodic orbit, and $D_2$ around the unstable focus for $\alpha
=0.03,~\rho=0.01$
\end{center}}}
%EndExpansion
\]

\[%
%TCIMACRO{\FRAME{itbpFU}{11.2439cm}{7.4839cm}{0cm}{\Qcb{Fig. 5: Phase portrait
%and control sets $D_0$ around the saddle, $D_1$ around the periodic orbit, and
%$D_2$ around the unstable focus for $\alpha=0.07,~\rho=0.01$}}{}%
%{figure5.eps}{\special{ language "Scientific Word";  type "GRAPHIC";
%width 11.2439cm;  height 7.4839cm;  depth 0cm;  original-width 0pt;
%original-height 0pt;  cropleft "0";  croptop "1";  cropright "1";
%cropbottom "0";  filename 'Figure5.eps';file-properties "XNPEU";}} }%
%BeginExpansion
\raisebox{-0cm}{\parbox[b]{11.2439cm}{\begin{center}
\includegraphics[
height=7.4839cm,
width=11.2439cm
]%
{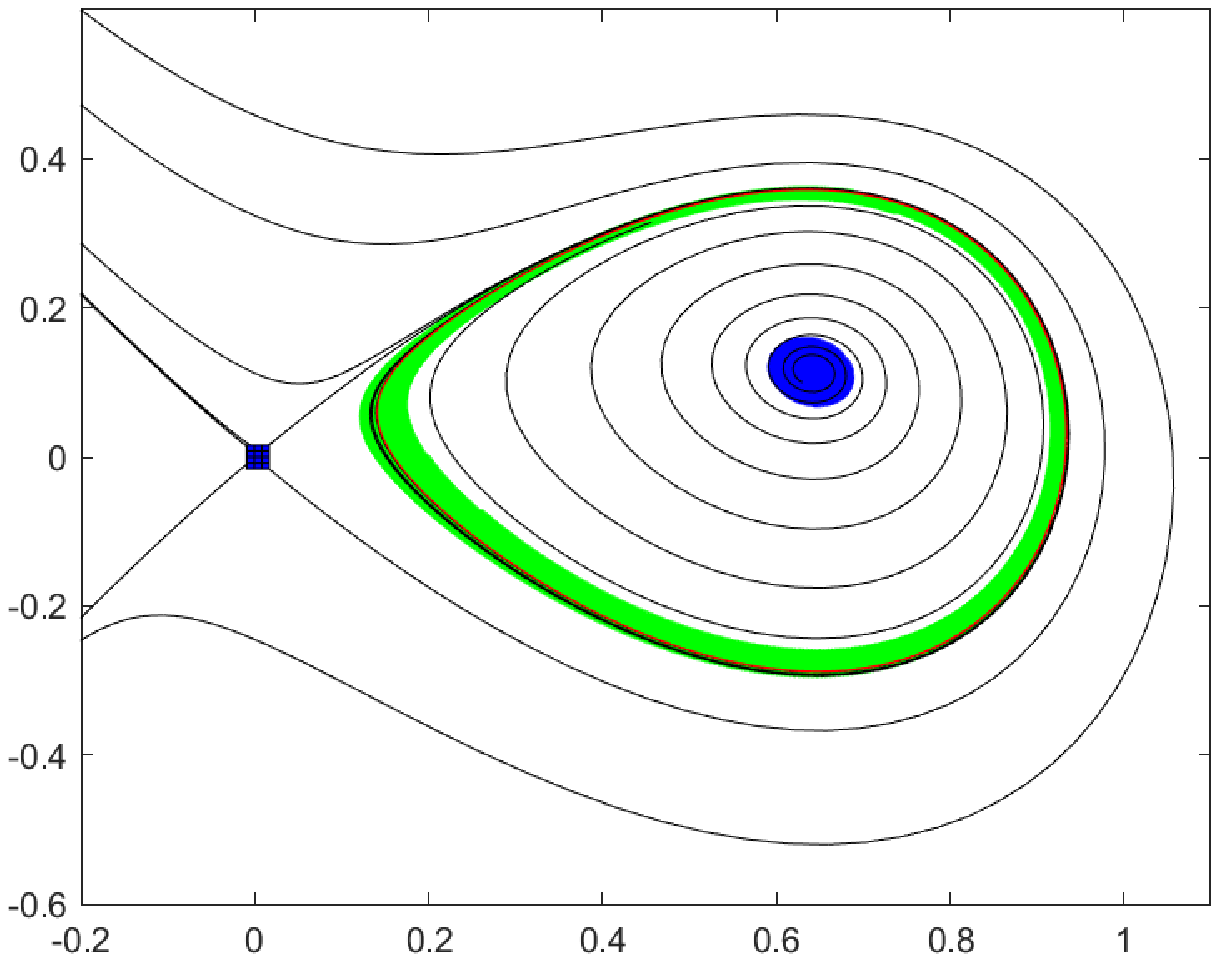}%
\\
Fig. 5: Phase portrait and control sets $D_0$ around the saddle, $D_1$ around
the periodic orbit, and $D_2$ around the unstable focus for $\alpha
=0.07,~\rho=0.01$
\end{center}}}
%EndExpansion
\]

\end{document}